\documentclass[12pt, twoside, leqno]{article}
\usepackage{amsmath,amsthm}
\usepackage{amssymb,latexsym}
\usepackage{graphicx}
\usepackage{enumerate}

\def\tr{\triangleright}

\newtheorem{theorem}{Theorem}
\newtheorem{definition}{Definition}
\newtheorem{lemma}[theorem]{Lemma}
\newtheorem{proposition}[theorem]{Proposition}
\newtheorem{corollary}[theorem]{Corollary}
\newtheorem{example}{Example}
\newtheorem{remark}[example]{Remark}

\begin{document}

\date{}

\title{Link invariants from finite racks}

\author{Sam Nelson\\
Department of Mathematical Sciences\\ 
Claremont McKenna College \\
850 Colubmia Ave., Claremont, CA 91711 \\
Email address: knots@esotericka.org}

\maketitle

\renewcommand{\thefootnote}{}



\renewcommand{\thefootnote}{\arabic{footnote}}
\setcounter{footnote}{0}

\begin{abstract}
We define ambient isotopy invariants of oriented knots and links using the 
counting invariants of framed links defined by finite racks. These
invariants reduce to the usual quandle counting invariant when the
rack in question is a quandle. We are able to further enhance these
counting invariants with 2-cocycles from the coloring rack's second rack
cohomology satisfying a new degeneracy condition which reduces to the
usual case for quandles.
\end{abstract}

\section{Introduction}

A \textit{rack} is a generally non-associative algebraic structure 
whose axioms correspond to blackboard-framed isotopy moves on link diagrams. 
Racks generalize \textit{quandles}, an algebraic structure whose axioms 
correspond to the three Reidemeister moves which combinatorially encode
ambient isotopy of knot diagrams.

Given a finite quandle $T$, the set of quandle homomorphisms from
a knot quandle $Q(K)$ to $T$ gives us an easily computed knot 
invariant, namely its cardinality $|\mathrm{Hom}(Q(K),T)|$. This is
the \textit{quandle counting invariant}, also sometimes called the
\textit{quandle coloring invariant} since each homomorphism $f:Q(K)\to T$ 
can be pictured as a ``coloring'' of the knot diagram assigning to 
each arc $x_i$ in a knot diagram the element $f(x_i)\in T$ such that 
a quandle coloring condition is satisfied at every 
crossing. Indeed, Fox 3-coloring is the simplest non-trivial example
of a quandle coloring invariant for knots.

If $T$ is a non-quandle rack, the set of colorings of arcs of a link
diagram by elements of $T$ satisfying the coloring condition at every 
crossing is invariant only under blackboard-framed isotopy -- type I 
Reidemeister moves which change the framing of the knot also change 
the number of colorings.  In this paper we will exploit a property
of finite coloring racks to define computable invariants
of ambient isotopy of knots and links incorporating these framed isotopy 
coloring invariants. The usual quandle coloring invariants
then form a special case of these more general rack coloring invariants.

The paper is organized as follows. In section \ref{fr} we review the
basics of racks, framed links and virtual links. In section \ref{cinv} we 
define finite rack based counting invariants and give some examples. In 
particular, we show that the writhe-enhanced invariant specializes 
to the integral invariant and contains more information.  
In section \ref{rh} we enhance the rack counting invariants with 2-cocycles 
in the style of \cite{CJKLS}. We provide an example showing that the 
cocycle-enhanced invariant contains more information than the writhe-enhanced
and integral rack counting invariants. In section \ref{q} we collect 
questions for future research.

\section{Basic definitions}\label{fr}

In this section we review the basic definitions we will need for the
remainder of the paper. 

\subsection{Racks}

We begin with a definition from \cite{FR}.

\begin{definition}
\textup{A \textit{rack} is a set $R$ with a binary operation 
$\tr:R\times R\to R$ satisfying 
\begin{list}{}{}
\item[(i)]{for all $x\in R$, the map $f_x:R\to R$ defined by
$f_x(y)=y\tr x$ is invertible, with inverse $f^{-1}_x(y)$ denoted
$y\tr^{-1}x$, and}
\item[(ii)]{for all $x,y,z\in R$, we have 
$(x\tr y)\tr z=(x\tr z) \tr (y\tr z)$.}
\end{list}
A rack in which $x\tr x= x$ for all $x\in R$ is a \textit{quandle.} The
operation $\tr^{-1}$ is the \textit{dual} rack operation -- it is also
self-distributive, and the two operations are mutually distributive.
Note that in \cite{FR}, $x\tr y$ is denoted $x^y$ and $x\tr ^{-1} y$ is
denoted $x^{\overline{y}}$.}
\end{definition}

\[\includegraphics{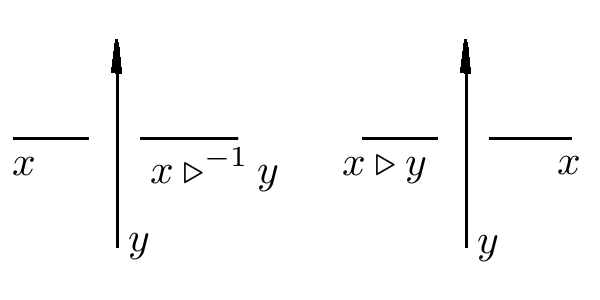}\] 

The rack axioms correspond to Reidemeister moves II and III where
we think of rack elements as arcs in an oriented link diagram and $\tr$ 
means crossing under from right to left when looking in the positive
direction of the overcrossing strand. The dual operation $\tr^{-1}$ can 
then be interpreted as crossing under from left to right.

\[\scalebox{1.2}{\includegraphics{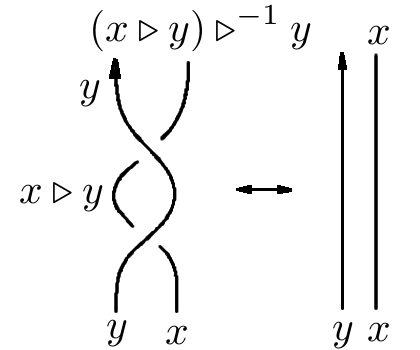}} \quad 
\scalebox{1.1}{\includegraphics{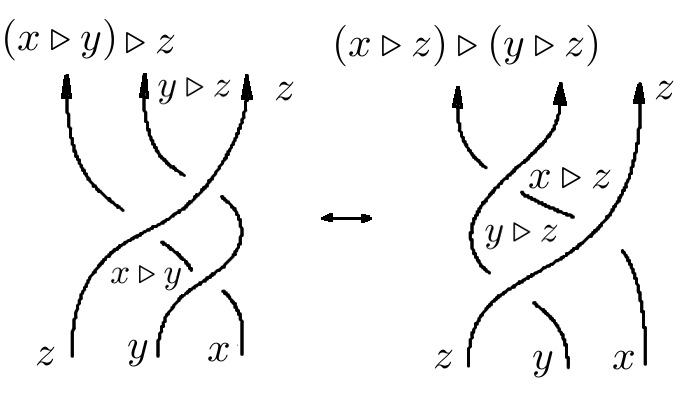}}\]

\begin{example}
\textup{Perhaps the simplest example of a rack structure on a finite set
$R=\{x_1,\dots,x_n\}$ is the \textit{constant action rack} or 
\textit{permutation rack} on $R$ associated to a permutation $\sigma\in S_n$. 
Specifically, set}
\[x_i\tr x_j = x_{\sigma(i)}\]
\textup{for all $i=1,\dots,n$; then the action of $y\in R$ on $R$ remains
constant as $y$ varies. It is easy to verify that this definition
gives us a rack structure, since $x_i\tr^{-1} x_j = x_{\sigma^{-1}(i)}$
and we have}
\[(x_i\tr x_j) \tr x_k = x_{\sigma^2(i)} = (x_i\tr x_k) \tr (x_j\tr x_k).\]
\textup{If a constant action rack is a quandle, then we have $x\tr x=x$ and
consequentially $x\tr y=x$ for all $x,y\in R$; such a quandle is called
\textit{trivial}. There is one trivial quandle for each cardinality
$n$, denoted $T_n$. We will denote the constant action rack associated
to $\sigma\in S_n$ by $T_{\sigma}$.}
\end{example}

\begin{example}
\textup{A simple example of a nontrivial rack structure from \cite{FR}
is the \textit{$(t,s)$-rack} structure: let $\ddot{\Lambda}$ be the ring
$\mathbb{Z}[t,t^{-1},s]$ modulo the ideal generated by $s^2-(1-t)s$. 
Then any $\ddot{\Lambda}$-module $M$ is a rack under the operation}
\[x\tr y = tx + sy.\]
\textup{For instance, we can take $M=\mathbb{Z}_n$ and choose $t,s\in M$
such that $\mathrm{gcd}(n,t)=1$ and $s^2=(1-t)s$, e.g. $M=\mathbb{Z}_8$ 
with $t=3$ and $s=2$. If $s=1-t$ then $M$ is a quandle, known as an 
\textit{Alexander quandle}.}
\end{example}

One useful way to describe a rack operation $\tr$ on a finite set 
$\{x_1,\dots, x_n\}$
is to encode its operation table as a matrix $M$ whose entry in row $i$
column $j$ is $k$ where $x_k=x_i\tr x_j$. Thus, the constant action rack
on $R=\{x_1,x_2,x_3\}$ defined by $\sigma = (123)$ has matrix
\[M_{(123)}=
\left[\begin{array}{ccc}
2 & 2 & 2 \\
3 & 3 & 3 \\
1 & 1 & 1
\end{array}\right]\]
and the $(t,s)$-rack $M=\mathbb{Z}_8$ with $t=3$ and $s=2$ has
rack matrix
\[
M_{(\mathbb{Z}_8,3,2)}=
\left[\begin{array}{cccccccc}
5& 7& 1& 3& 5& 7& 1& 3 \\ 
8& 2& 4& 6& 8& 2& 4& 6 \\ 
3& 5& 7& 1& 3& 5& 7& 1 \\
6& 8& 2& 4& 6& 8& 2& 4 \\
1& 3& 5& 7& 1& 3& 5& 7 \\
4& 6& 8& 2& 4& 6& 8& 2 \\
7& 1& 3& 5& 7& 1& 3& 5 \\
2& 4& 6& 8& 2& 4& 6& 8
\end{array}\right].
\]
Rack axiom (i) requires the columns of a rack matrix to be
permutations. See \cite{HN} for more.

\subsection{Framed links}

Recall that a \textit{framed link} is a link $L$ with a choice of 
\textit{framing curve} $F_i$ for every component $C_i$ of $L$, i.e. 
$F_i$ is a longitude of a regular neighborhood of $C_i$. Framing 
curves are determined up to isotopy by their linking numbers with 
$C_i$. In terms of diagrams, we can bestow a canonical framing on 
every component of a link via the \textit{blackboard framing}, i.e. 
drawing a framing curve for each $C_i$ parallel to $C_i$. This gives
a framing with linking number given by the \textit{writhe}
$w(C_i)=\sum_{x\in S_i} \mathrm{sign}(x)$ 
where $S_i$ is the set of crossings where $C_i$ crosses itself,
$\mathrm{sign}\left(\raisebox{-0.05in}{\scalebox{0.5}{\includegraphics{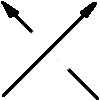}}}
\right)=1$ and $\mathrm{sign}\left(\raisebox{-0.05in}{\scalebox{0.5}{
\includegraphics{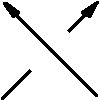}}}\right)=-1.$

Combinatorially, blackboard-framed links can be regarded as equivalence 
classes of link diagrams under the equivalence relation generated by 
Reidemeister moves II and III together with a doubled type I move
which preserves the framing of each component; see \cite{T,FR}.

\[\includegraphics{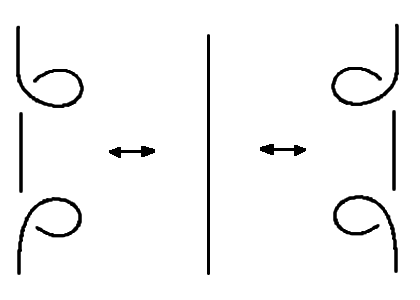} \quad \quad
\includegraphics{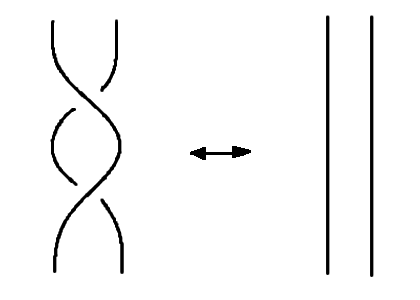} \quad \quad
\includegraphics{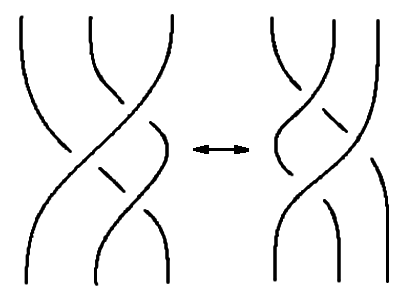}\]

\subsection{Virtual links}

Virtual knot theory is a combinatorial generalization of ordinary
\textit{classical} knot theory; geometrically, a virtual link is
an ordinary link in which the ambient space is not $S^3$ or 
$\mathbb{R}^3$ but $\Sigma\times [0,1]$ for some compact orientable surface 
$\Sigma$, considered up to stabilization (see \cite{K, CKS}). 
More formally, we have:

\begin{definition}
\textup{A \textit{virtual link} is an equivalence class of link diagrams
with an extra crossing type known as a \textit{virtual crossing},
\includegraphics{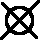}, 
under the equivalence relation determined by the usual Reidemeister moves
together with the four \textit{virtual moves}}
\[\scalebox{0.95}{\includegraphics{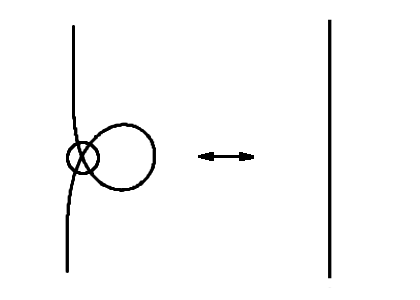}}\quad 
\scalebox{0.95}{\includegraphics{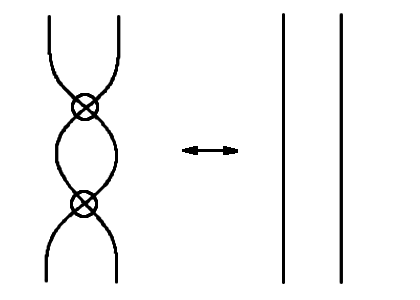}} \]\[
\includegraphics{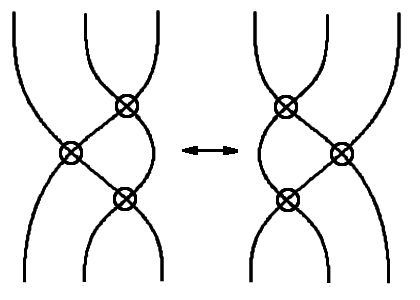} \quad 
\includegraphics{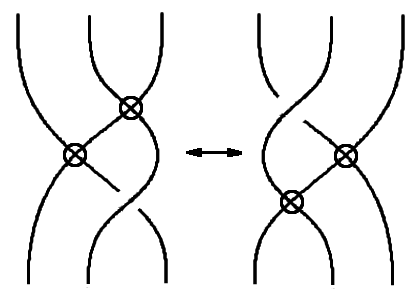}\]
\end{definition}

We can summarize the rules for virtual moves with the \textit{detour move},
which says that any strand with only virtual crossings can be replaced
by any other strand with the same endpoints with only virtual crossings.
That is, a strand with only virtual crossings can move past any virtual
tangle.

\[\includegraphics{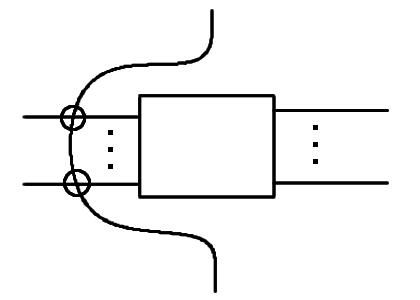} \quad \raisebox{0.5in}{$\longleftrightarrow$}
\quad \includegraphics{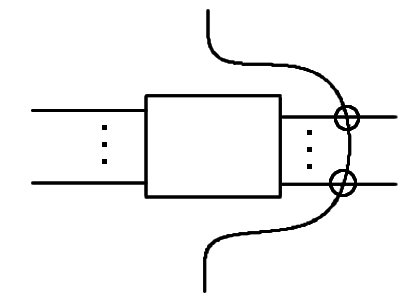}\]

Virtual crossings have no intrinsic over- or under-sense, as they are 
artifacts of drawing non-planar link diagrams on planar paper. Classical
links are then virtual links whose underlying surface $\Sigma$ is $S^2$.
Replacing the classical Reidemeister I move with the doubled version
yields \textit{framed virtual links}. For the remainder of this paper,
we will use ``link diagram'' to mean ``oriented blackboard-framed virtual or 
classical link diagram.'' See \cite{K} for more.

\subsection{\normalsize \textbf{The fundamental rack(s) of a link}}

Associated in \cite{FR} to a framed link $L$ is a rack known as the 
\textit{fundamental rack} of $L$, which we will denote by 
$FR(L)$.\footnote{Does the notation
``FR'' stand for ``fundamental rack'' or ``Fenn and Rourke''? Perhaps both!} 

Geometrically, elements of $FR(L)$ are homotopy classes of paths in the
link complement $X=S^3\setminus (L\times \mathrm{Int}(B^2))$ from the framing 
curves $\cup F_i\subset \partial(X)$ to a fixed base point $x_0\in X$ where the 
terminal point is fixed but the initial point is permitted to wander along 
the framing curve $F_i$ during the homotopy. Any such path $\alpha:[0,1]\to X$ 
has an associated element $\pi(\alpha)$ of the fundamental group $\pi_1(X,x_0)$
defined by traveling backwards along $\alpha,$ then going around the 
canonical meridian 
in $\partial(X)$ intersecting $\alpha(0)$, then going back along $\alpha$. 
The rack operation is then 
\[[\beta]\tr [\alpha] = [\beta \ast \pi(\alpha)]\]
where $\ast$ is concatenation of paths. 
\[\includegraphics{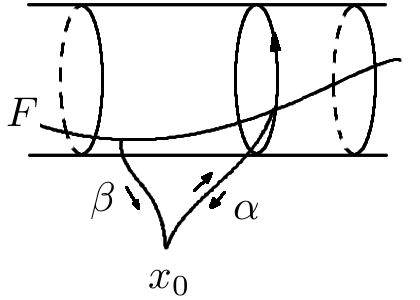}\]

Combinatorially, given a diagram of $L$, the fundamental rack of $L$ consists 
of equivalence classes of rack words in generators corresponding to arcs in 
the diagram of $L$ under the equivalence relation generated by the rack axioms
together with the relations imposed at each crossing. If $L$ is a virtual 
link, we simply ignore the virtual crossings. 

\begin{example}\label{vt1}
\textup{The pictured blackboard-framed virtual link  has 
fundamental rack with generators $x, y$ and relation 
$x\tr y = x \tr (y\tr x)$:}

\noindent
\raisebox{-0.5 in}{\includegraphics{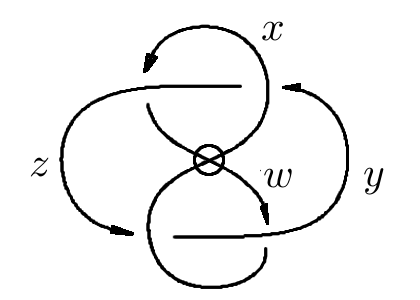}} 
\parbox{3in}{\scalebox{0.9}{$ \begin{array}{lll}
FR(L) & = & \langle x,\ y,\ z,\ w \ | \ y\tr x= z, \ x\tr z = w, x\tr y = w
\rangle \\
 & = & \langle x,\ y, \ w \ | \ x\tr y = w, \  x \tr (y\tr x) =w \rangle \\
 & = & \langle x,\ y \ | \ x\tr y = x \tr (y\tr x) \rangle \\
\end{array}
$}}
\end{example}

For each framing of a given link, we have a fundamental rack, generally
distinct from the racks of the other framings. All of these racks have
a common quotient quandle obtained by setting $a\tr a = a$ for all elements
$a\in FR(L)$, which is the \textit{knot quandle} $Q(L)$ of the unframed 
link $L$. Elements of the knot quandle may be interpreted geometrically as
homotopy classes of paths where the initial point is permitted to wander
not just along the framing curve but along all of $\partial(X)$.
See \cite{FR} for more.

\section{Racks and counting invariants}\label{cinv}

Let $L$ be an unframed link with an ordering on the components. If $L$ has 
$n$ components, then the
framings of $L$ may be indexed by $n$-tuples $\mathbf{w} \in \mathbb{Z}^n$, 
each with its own \textit{a priori} distinct fundamental rack.
At the most basic level, then, there are infinitely many rack counting 
invariants for a given link with respect to any choice of finite target
rack $T$. However, we can make a useful observation which enables us to get
computable ambient isotopy invariants from the $\mathbb{Z}^n$-set of 
racks of $L$.

\begin{definition}
\textup{Let $T$ be a rack. For any $x\in T$, let $x^{\tr n}$ 
for $n\in \mathbb{Z}_+$ be defined recursively by}  
\[x^{\tr 1} = x\tr x \quad \mathrm{and} 
\quad x^{\tr (k+1)}= x^{\tr k} \tr x^{\tr k}.\] 
\textup{For each element $x\in T$, the \textit{rack rank} of $x\in T$,
denoted $N(x)$, is the minimal natural number $N\in\mathbb{Z}_+$ such that
$x^{\tr N}=x$, or $N(x)=\infty$ if $x^{\tr N}\ne x$ for all $N\in \mathbb{N}$. 
The \textit{rack rank} of $T$, denoted $N(T)$ or just
$N$ if $T$ is understood, is the least common multiple of the rack ranks
of the elements of $T$,} 
\[N(T)=\mathrm{lcm}\{N(x)\ |\ x\in T\}.\]
\end{definition}

To see that $N(x)$ is well defined for all $x\in T$, we first need a lemma.

\begin{lemma}\label{lem:opeq}
Let $T$ be any rack and $x,y\in T$. Then $y\tr(x\tr x) = y\tr x$.
\end{lemma}

\begin{proof}
\[y\tr(x\tr x) = [(y\tr^{-1} x)\tr x]\tr (x\tr x)
= [(y\tr^{-1}x)\tr x]\tr x =y\tr x.\]
\end{proof}

\begin{remark}
\textup{Two elements $x,y\in T$ are \textit{operator equivalent} if
$z\tr x=z\tr y$ for all $z\in T$. If $T$ is a finite rack, then
two elements are operator equivalent iff their columns in the matrix of $T$ 
are identical.
Lemma \ref{lem:opeq} says that the $\tr$-powers of $x\in T$ are
all operator equivalent. Indeed, the set of operator equivalence classes
of a rack forms a quandle under the natural operation 
$[x]\tr[y]=[x\tr y]$.}
\end{remark}

\begin{corollary}\label{diagperm}
Let $T$ be a rack. If $x\tr x=y\tr y$, then $x=y$.
\end{corollary}

\begin{proof}
Suppose $x\tr x=y\tr y=z.$ We have
$x\tr x = x\tr (x\tr x)=x\tr z$
and $y\tr y = y\tr (y\tr y) = y\tr z.$
Then $x\tr x=y\tr y$ implies $x\tr z = y\tr z$  and rack axiom (i)
implies $x=y$.
\end{proof}

In terms of rack matrices, corollary \ref{diagperm} says that like the
columns of a rack matrix, the diagonal of a rack matrix must be a permutation.
Indeed, if we define $\pi:T\to T$ by $\pi(x)=x\tr x$ then the diagonal
of a rack matrix tell us the permuation $\pi$.
It then easily follows that $N(x)<|T|$ for any $x\in T$ where $T$ is a 
finite rack -- indeed, $N(T)$ is just the exponent of $\pi\in S_{|T|}$. 
This fact also follows from proposition 7.3 in \cite{LR}. 

We will also need the following standard result (see \cite{FR} or \cite{T} 
for example):

\begin{theorem}\label{thm:fld}
If $D$ and $D'$ are ambient isotopic link diagrams, we can modify $D'$ 
to obtain a diagram $D''$ which is framed isotopic to $D$ by selecting 
an arc on each component of $D'$ and adding positive or negative kinks
until the framings match.
\end{theorem}

The proof of theorem \ref{thm:fld} involves taking any Reidemeister
move sequence starting with $D$ and ending with $D'$ and replacing every
type I move with a double I move to adjust the framed isotopy class; at the end,
we can then slide the extra crossings along the component until they arrive at
the chosen arc. Note that this argument applies to virtual links as well as 
classical links, since we can slide a classical kink past a virtual crossing 
using a detour move. Note also that without loss of generality we can assume 
that all kinks added have positive winding number since we need not preserve
the regular isotopy class, only the blackboard-framed isotopy class.

\begin{definition}\textup{
Let $N\in \mathbb{N}$. We say two blackboard-framed oriented link diagrams 
are} \textit{$N$-phone cord 
equivalent} \textup{if one may be obtained from other by a finite sequence of
Reidemeister II and III moves and the following \textit{$N$-phone cord} move,
where $N$ is the number of loops:}
\[\includegraphics{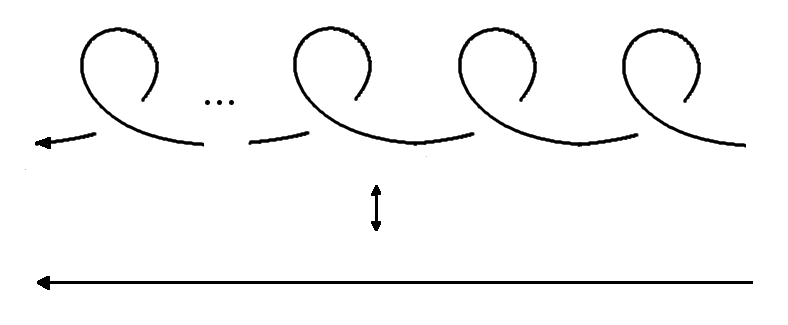}\]
\end{definition}

\begin{proposition}
Let $T$ be a finite rack with rack rank $N$. If two link diagrams $D$ and 
$D'$ are $N$-phone cord isotopic then 
$|\mathrm{Hom}(FR(D),T)|=|\mathrm{Hom}(FR(D'),T)|$.
\end{proposition}

\begin{proof}
From the definition of rack rank, it is easy to see that $N$-phone cord moves
induce a bijection on the set of colorings as illustrated.
\[\includegraphics{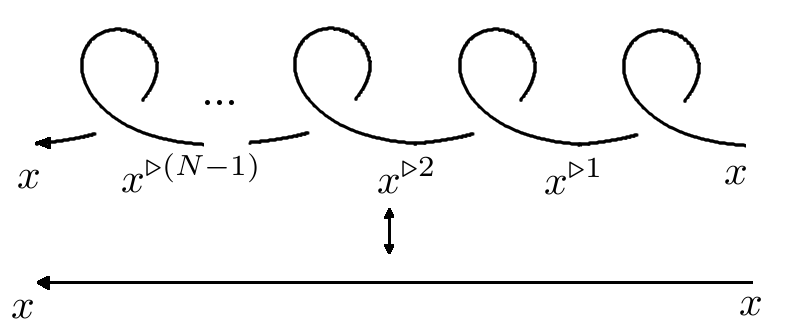}\]
\end{proof}

For two $n$-tuples $\mathbf{v},\ \mathbf{w}\in \mathbb{Z}^n$, let us write
$\mathbf{v}\equiv \mathbf{w}\ \mathrm{mod} \ N$ if for all components
$i=1,\dots n$ we have $v_i\equiv w_i \ \mathrm{mod} \ N$.

\begin{corollary}
Let $T$ be a finite rack with rack rank $N$.
If two link diagrams $D$ and $D'$ are ambient isotopic and have 
writhe vectors congruent modulo $N$, 
then $|\mathrm{Hom}(FR(D),T)|=|\mathrm{Hom}{FR(D'),T}|$.
\end{corollary}

Note that if $T$ is a finite rack with rack rank $N$ and $L$ is a link, 
the set of writhes of each component of $L$ modulo $N$ can be indexed 
by $\mathbf{w}\in (\mathbb{Z}_N)^c$ where $c$ is the number of 
components of $L$. For ease of notation, when $N$ and $c$ are understood 
let us denote $(\mathbb{Z}_N)^c=W$ and a blackboard-framed 
diagram of $D$ with writhe vector $\mathbf{w}\in W$ by $(D,\mathbf{w}).$

\medskip

We can now define computable unframed knot and link invariants using these
cardinalities.

\begin{definition}
\textup{Let $T$ be a finite rack and $L$ a link with $c$ components. 
The \textit{integral rack counting invariant} of $L$ with respect to $T$ is}
\[\Phi^{\mathbb{Z}}_T(L) = \sum_{\mathbf{w}\in W} 
|\mathrm{Hom}(FR(D,\mathbf{w}),T)|.\]
\end{definition}

Note that if $T$ is a quandle, then we have $N(T)=1$ and 
$\Phi^{\mathbb{Z}}_T(L)$ is
the ordinary quandle counting invariant $|\mathrm{Hom}(Q(L),T)|$. Hence
the integral rack counting invariant is the natural generalization of the
quandle counting invariant to the finite rack case.

\begin{example}
\textup{If $T=\{x_1,\dots, x_n\}$ is a constant action rack defined by an 
$n$-cycle, an undercrossing color $\tau$ becomes $\sigma(\tau)$ 
if going right-to-left and $\sigma^{-1}(\tau)$ if going left-to-right,
so pushing a color around the knot yields an ending color of 
$\sigma^{w(K)}(\tau)$ where $w(K)$ is the writhe of $K$ and our starting 
color was $\tau$. Thus, there is a rack coloring of a framed knot $K$ by 
$T$ if and only if the writhe of $K$ is zero mod $n$. Indeed, there are 
$n$ such colorings for the $0$-framing mod $n$ and none for the others, 
and we have $\Phi^{\mathbb{Z}} _T(K)=n+(n-1)0=n$ for any knot 
$K$. This generalizes the fact that $|\mathrm{Hom}(K,T)|=n$ for $T$ a 
trivial quandle of cardinality $n$ and $K$ a knot.}
\end{example}

\begin{example}
\textup{Let $T_{\ast}$ be the rack with matrix
\[M_{T_{\ast}}\left[\begin{array}{ccccccc}
1 & 3 & 2 & 1 & 1 & 1 &  1  \\
3 & 2 & 1 & 2 & 2 & 2 & 2  \\
2 & 1 & 3 & 3 & 3 & 3 & 3  \\
4 & 4 & 4 & 6 & 4 & 6 & 4  \\
5 & 5 & 5 & 5 & 7 & 5 & 7  \\
6 & 6 & 6 & 4 & 6 & 4 & 6  \\
7 & 7 & 7 & 7 & 5 & 7 & 5  
\end{array}\right].\] The integral rack counting invariant with respect to 
$T_{\ast}$ distinguishes the trefoil $3_1$ from the unknot $U_1$ with 
$\Phi^{\mathbb{Z}}_{T_{\ast}}(U_1))=10$ and $\Phi^{\mathbb{Z}}_{T_{\ast}}(3_1)=22$, 
as the reader can verify from the tables of colorings listed in table 1. 
Here $N(T_{\ast})=2$, so we 
need only consider one diagram each of $U_1$ and $3_1$ with odd writhe and 
one of each with even writhe.}
\end{example}

\begin{table}[!ht]
\begin{center}
\begin{tabular}{|c|l|} \hline
$D$ & Colorings by $T_{\ast}$ \\ \hline 
\raisebox{-0.5in}{\includegraphics{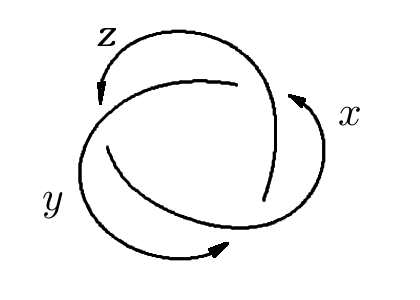}} & $
\begin{array}{|c|ccccccccc|} \hline
x & 1 & 1 & 1 & 2 & 2 & 2 & 3 & 3 & 3 \\
y & 1 & 2 & 3 & 1 & 2 & 3 & 1 & 2 & 3 \\
z & 1 & 3 & 2 & 3 & 2 & 1 & 2 & 1 & 3 \\ 
\hline \end{array}$
 \\ \hline
\raisebox{-0.5in}{\includegraphics{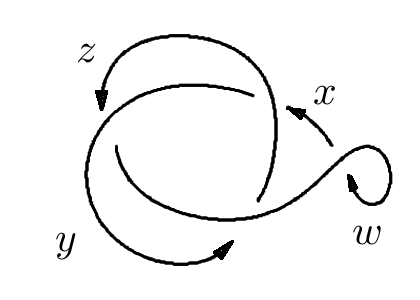}} & 
$\begin{array}{|c|ccccccccccccc|} \hline
x & 1 & 1 & 1 & 2 & 2 & 2 & 3 & 3 & 3 & 6 & 4 & 5 & 7 \\
y & 1 & 2 & 3 & 1 & 2 & 3 & 1 & 2 & 3 & 4 & 6 & 7 & 5 \\
z & 1 & 3 & 2 & 3 & 2 & 1 & 2 & 1 & 3 & 6 & 4 & 5 & 7 \\ 
w & 1 & 1 & 1 & 2 & 2 & 2 & 3 & 3 & 3 & 4 & 6 & 7 & 5 \\ \hline
\end{array}$ \\ \hline
\raisebox{-0.5in}{\includegraphics{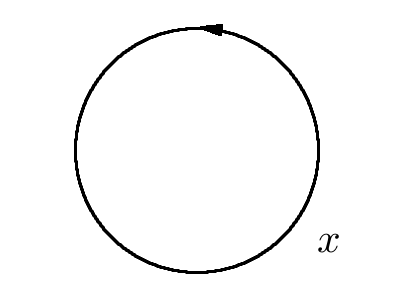}} & 
$\begin{array}{|c|ccccccc|} \hline
x & 1 & 2 & 3 & 4 & 5 & 6 &7 \\ \hline \end{array}$ \\  \hline 
\raisebox{-0.5in}{\includegraphics{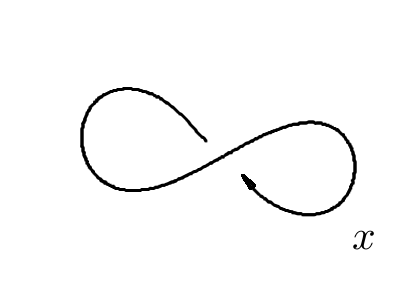}} & 
$\begin{array}{|c|ccc|} \hline
x & 1 & 2 & 3 \\ \hline \end{array}$ \\ \hline
\end{tabular}
\end{center}
\caption{Rack colorings of $3_1$ and $U_1$ by $T_{\ast}$}
\end{table}

We can enhance the integer rack counting invariant by keeping track of which
framings contribute which colorings. For a writhe vector 
$\mathbf{w}=(w_1,\dots,w_c)\in W=(\mathbb{Z}_N)^c$ let us denote the product 
$\prod_{k=1}^c q_1^{w_1}q_2^{w_2}\dots q_c^{w_c}$ by $q^{\mathbf{w}}$. Then we have:

\begin{definition}
\textup{Let $T$ be a finite rack and $L$ a link with $c$ components
and writhe vector $\mathbf{w}=(w_1,\dots,w_c)\in W$. 
The \textit{writhe-enhanced rack counting invariant} of $L$ with respect to 
$T$ is given by}
\[\Phi^{W}_{T}(L) = \sum_{\mathbf{w}\in W}
|\mathrm{Hom}(FR(L,\mathbf{w}),T)|q^{\mathbf{w}}. \]
\end{definition}

The writhe-enhanced rack counting invariant holds more information
than the simple version, as the next example shows.

\begin{example}
\textup{The constant action rack $T$ with rack matrix 
$\left[\begin{array}{cc}
2 & 2 \\
1 & 1 \\
\end{array}\right]$ has rack rank $N(T)=2$. The Hopf link $H$ and the 
two-component unlink $U_2$ both have integral rack counting invariant value
$\Phi^{\mathbb{Z}}_{T}(L)=4$ with respect to $T$, but the writhe-enhanced 
rack counting invariants are distinct, with $\Phi^{W}_{T}(H)=4q_1q_2$ 
and $\Phi^{W}_{T}(U_2)=4$ as the reader can easily verify from table 2.}
\end{example}

\begin{table}[!ht]
\[
\begin{array}{|cccc|}\hline
\includegraphics{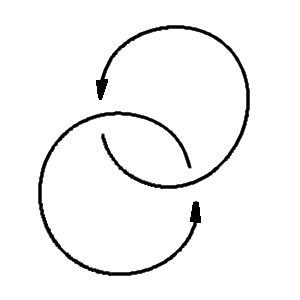} & \includegraphics{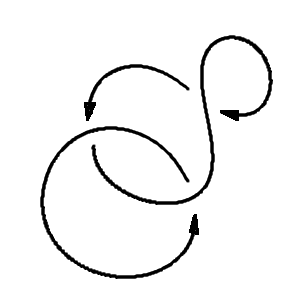} &
\includegraphics{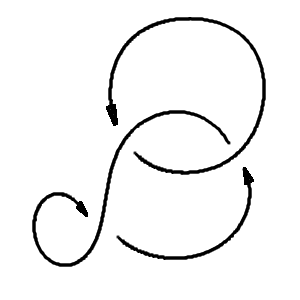} & \includegraphics{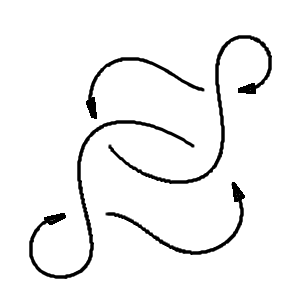} \\
 0 & 0 & 0 & 4 \\ \hline
\includegraphics{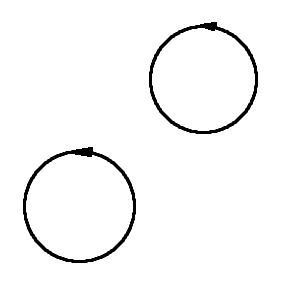} & \includegraphics{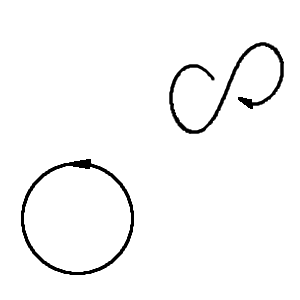} &
\includegraphics{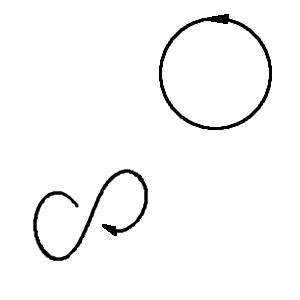} & \includegraphics{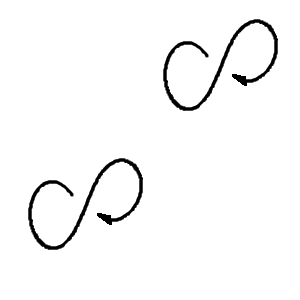} \\
 4 & 0 & 0 & 0 \\ \hline
\end{array}\]
\caption{Numbers of colorings of $H$ and $U_2$ by $T_{(12)}$.}
\end{table}

Indeed, generalizing the preceeding example we have
\begin{proposition}
Let $L$ be a two-component classical link and $T=T_{\sigma}$ a constant 
action rack with $\sigma\in S_N$ an $N$-cycle. Then the writhe-enhanced rack 
counting invariant has the form
\[\Phi^{W}_{T_{\sigma}}(L)=N^2q_1^lq_2^l\]
where $l$ is the negative of the linking number $lk(L_1,L_2)$ of $L$ mod $N$.
\end{proposition}

\begin{proof}
Traveling around a component, to get a valid coloring the end color must 
match the initial color, so we must go through $N$ crossings (counted 
algebraicially). Since $lk(L_1,L_2)$ of these are multi-component crossings
which do not contribute to the component's writhe, we must have 
\[l+lk(L_1,L_2)=N.\]
The same holds for both components if $L$ is classical. There are $N$ choices
of starting color for each component and every pair produces exactly one 
coloring, so there are $N^2$ total colorings.
\end{proof}

\begin{corollary}
If $L$ is a two-component classical link and $T=T_{\sigma}$ a constant 
action rack with $\sigma\in S_N$ an $N$-cycle such that the exponents of
of $q_1$ and $q_2$ differ in any term of $\Phi^{W}_{T_{\sigma}}(L)$, then $L$ 
is non-classical.
\end{corollary}

\section{Rack cocycle invariants}\label{rh}

In this section we generalize the quandle 2-cocycle invariants defined
in \cite{CJKLS} to the finite rack case.

The rack counting invariants described in the last section are
cardinalities of sets of homomorphisms which are unchanged by Reidemeister
moves. However, a set is more than a cardinality, and we would like to
recover as much information from these sets of homomorphisms as possible.

In \cite{CJKLS}, the idea is to associate a sum in an abelian group $A$ 
called a \textit{Boltzmann weight} to a quandle-colored knot diagram in 
such a way that the sum
does not change under Reidemeister moves. Then, instead of counting ``1'' for 
each homomorphism, we count its Boltzmann weight, transforming the set of 
colorings into a multiset of these weights. Such multisets are commonly
encoded as polynomials by taking a generating function, i.e. by converting 
the multiset elements to exponents and multiplicities to coefficients of a 
dummy variable, e.g. $\{1,1,1,4,4\}$ becomes $3t+2t^4$.

The Boltzmann weights are defined as follows: at every crossing in a 
rack-colored link diagram, we want to count $\phi(a,b)$ at a positive 
crossing or $-\phi(a,b)$ at a negative crossing where $b$ is the color on 
the overcrossing strand and $a$ is the color on the inbound understrand for 
positive crossings and the outbound understrand for negative crossings. 

\[\includegraphics{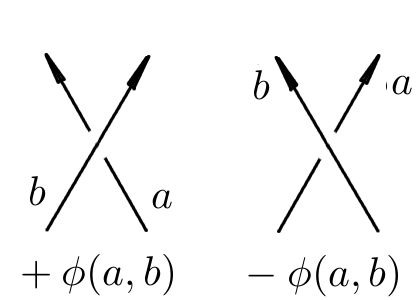}\]

This weighting rule has the advantage that the contributions from the two 
crossings in a Reidemeister type II move cancel, so the sum is automatically 
invariant under II moves:

\[\includegraphics{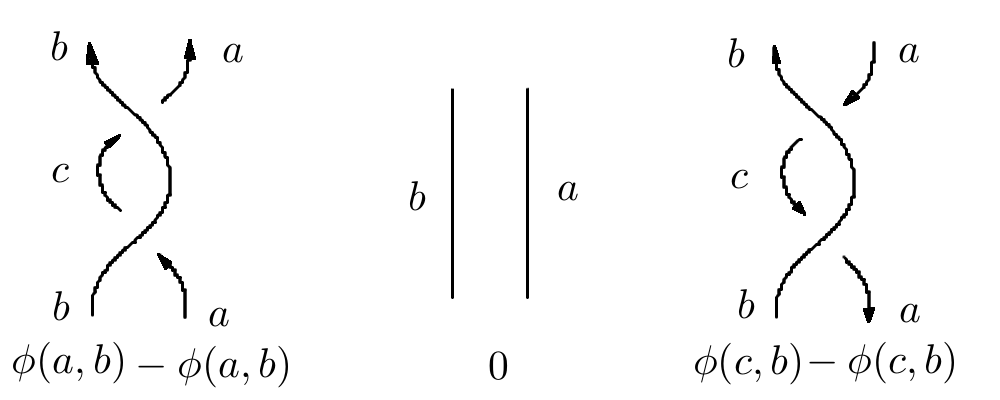}\]

We also note that the weighting rule gives invariance under the
doubled type I moves required for blackboard-framed isotopy:

\[\includegraphics{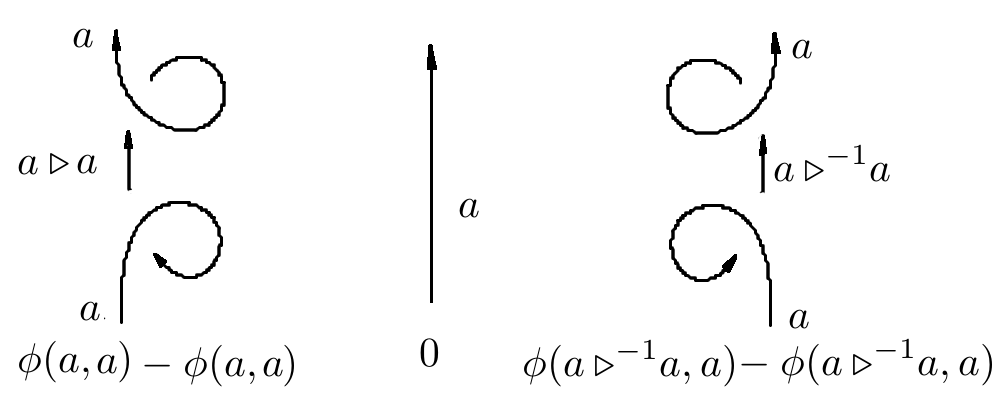}\]

The condition for the sum to be unchanged by Reidemeister III moves is 
pictured below.
\[\includegraphics{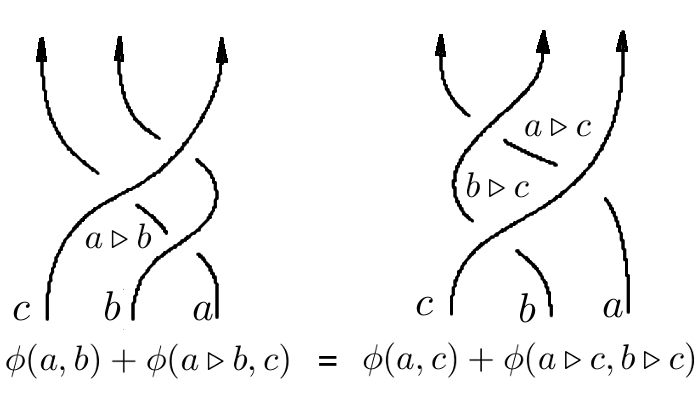}\]
This turns out to be the condition that $\phi$ is a cocycle in the second 
rack cohomology $H^2_R(T;A)$ of the rack $T$ with coefficients in $A$. 
Specifically, the $A$-module spanned by $n$-tuples of elements
of $T$ is the space of \textit{rack} $n$-\textit{chains} $C_n^R(T;A)=A[T^n]$;
its dual is the space of \textit{rack} $n$-\textit{cochains}
$C_R^n(T;A)=\mathrm{Hom}(C_n^R(T;A),A)$. Note that $C^n_R(T;A)$ has 
$A$-generating set $\{\chi_{\mathbf{x}}\ |\ \mathbf{x}\in T^n\}$ where 
\[\chi_{\mathbf{x}}(\mathbf{y}) =\left\{\begin{array}{ll}
1 & \mathbf{x}=\mathbf{y} \\
0 & \mathrm{otherwise}
\end{array}\right.\]
for $\mathbf{y}\in T^n$. Next, we define a coboundary map
$\delta^n:C_R^n(T;A)\to C_R^{n+1}(T;A)$ by
\begin{eqnarray*} (\delta^n\phi)(x_1,\dots,x_{n+1}) 
& = &  
\sum_{i=1}^{n+1} (-1)^{i-1}\left(\phi(x_1,\dots,x_{i-1},x_{i+1},
\dots,x_{n+1})\right. \\ 
& & \quad -\left.\phi(x_1\tr x_i,\dots,x_{i-1}\tr x_i,x_{i+1},
\dots,x_{n+1})\right).
\end{eqnarray*}
Then for $\phi$ to yield a Boltzmann 
weight, we need $\phi\in\mathrm{Ker}(\delta^2).$

To get invariance under the Reidemeister I move in the quandle case, we 
require that $\phi(x,x)=0$ for all $x\in T$. This condition also turns out
to have a homological interpretation -- the cocycles we want to kill live in
a subcomplex called the \textit{degenerate} cochains.
In the non-quandle rack case, however, a weaker condition suffices.

\begin{definition}
\textup{Let $T$ be a finite rack with rack rank $N$, $A$ an abelian group and 
$\phi\in C^2_R(T,F).$ Say $\phi$ is} $N$-reduced {if we have}
\[\sum_{k=1}^{N}\phi(a^{\tr k},a^{\tr k}) =0 \]
\textup{for all $a\in T.$}
\end{definition}

Now we can define an enhanced version of the polynomial rack counting 
invariant:

\begin{definition}
\textup{
Let $L$ be an oriented blackboard-framed link, $T$ a finite rack and $\phi$ an $N$-reduced rack 2-cocycle. For
a rack-colored framed diagram of $L$, $(D,\mathbf{w})$, the \textit{Boltzmann 
weight} $BW(f)$ of
the coloring $f\in \mathrm{Hom}(FR(D,\mathbf{w}),T)$ is the
sum over the crossings in $D$ of the crossing weights,}
\[BW(f)=\sum_{c\ \mathrm{crossing}} \mathrm{sign}(c)\phi(a,b).\] 
\textup{Then the \textit{rack cocycle invariant} of $L$ with respect to $T$ is}
\[\Phi^{\phi}_T(L)= 
\sum_{ \mathbf{w}\in W} \left(
\sum_{ f\in \mathrm{Hom}(FR(D,\mathbf{w}),T)}z^{BW(f)}\right)\]
\textup{and the \textit{writhe-enhanced rack cocycle invariant} is}
\[\Phi^{\phi,W}_T(L)= 
\sum_{ \mathbf{w}\in W} \left(
\sum_{ f\in \mathrm{Hom}(FR(D,\mathbf{w}),T)}z^{BW(f)}\right)q^{\mathbf{w}}.\]
\end{definition}

Note that if $T$ is a quandle, then $N(T)=1$ and $\phi$ is $1$-reduced iff 
$\phi(x,x)=0$ for all degenerate cycles $(x,x)\in C_2^R(T,A)$; in 
this case we also have $W=\{(0,\dots,0)\}$. Indeed,
in the quandle case this rack cocycle invariant becomes the usual CJKLS quandle
2-cocycle invariant from \cite{CJKLS}.

Specializing $z=1$ in $\Phi^{\phi,W}_T(L)$ recovers the writhe-enhanced rack
counting invariant, and likewise in the integral case. To see that the integral
rack cocycle invariant is stronger than $\Phi^{\mathbb{Z}}_T(L)$, consider the 
following simple example.

\begin{example}
\textup{The rack $T$ with rack matrix}
\[M_T=\left[\begin{array}{cccc}
3 & 1 & 3 & 1 \\
2 & 4 & 2 & 4 \\
1 & 3 & 1 & 3 \\
4 & 2 & 4 & 2
\end{array}\right]\]
\textup{has a reduced cocycle 
$\phi=\chi_{(12)}+\chi_{(14)}+\chi_{(32)}+\chi_{(34)}$ with 
$\mathbb{Z}_{13}$ coefficients. Then the $(4,2)$-torus link is distinguished
from the two-component unlink by $\Phi_{\phi}$:}
\[\begin{array}{cc}
\includegraphics{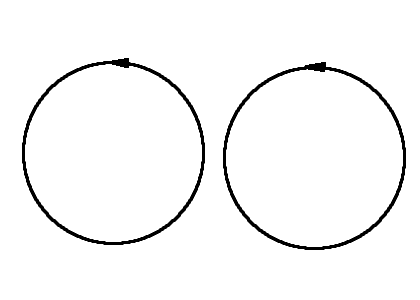} & \includegraphics{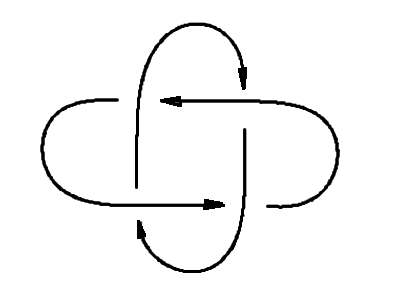} \\
\Phi_{\phi} = 16 & \Phi_{\phi}= 8+ 8z^{12}
\end{array}\] 
\end{example}

Our final example illustrates a pair of virtual links with equal 
writhe-enhanced rack counting invariant values which are distinguished when 
we include the rack cocycle information.

\begin{example}
\textup{Again let $T$ be the rack with rack matrix}
\[M_T=\left[\begin{array}{cccc}
3 & 1 & 3 & 1 \\
2 & 4 & 2 & 4 \\
1 & 3 & 1 & 3 \\
4 & 2 & 4 & 2
\end{array}\right]\] 
\textup{and $\phi=\chi_{(1,2)}+\chi_{(1,4)}
+\chi_{(3,2)}+\chi_{(3,4)}\in C^2_R(T;\mathbb{Z}_{13})$. Then 
$\Phi_{\phi}$ distinguishes the two 
pictured virtual links, both of which have $\Phi^{W}_T(L)=8+8q_1$. Note that 
the subscripts on $q$ correspond to the component ordering.}
\[\begin{array}{cc}
\includegraphics{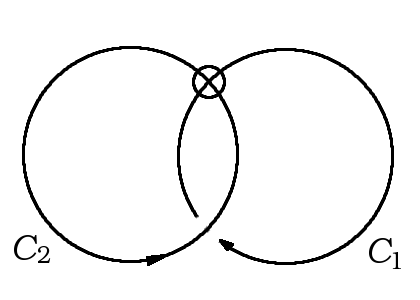} & \includegraphics{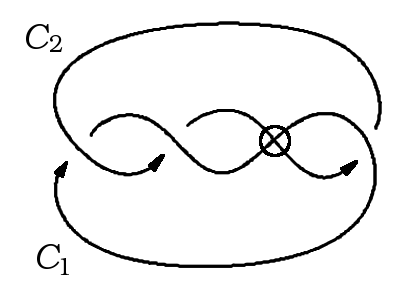} \\
\Phi^{\phi,W}_T = 4+4z+8q_1 & \Phi^{\phi,W}_T= 4z+4z^2+8q_1
\end{array}\]
\end{example}

For more on rack cocycle invariants, see \cite{EN}.

\section{Questions} \label{q}

In this section we collect a few questions for future research.

Rack and quandle (co)homology has been generalized in various ways 
including
\textit{twisted quandle (co)homology} in \cite{CES}, quandle (co)homology
with coefficients in quandle modules in \cite{CEGS} and more.
How does the rack cocycle invariant change in these cases?

Quandle 3-cocycles have been used to enhance quandle counting invariants
of surface knots, i.e. embeddings of compact orientable 2-manifolds in $S^4$.
How do the rack counting and cocycle invariants extend to the surface knot
case?

Other ways of enhancing the quandle counting invariants include using quandle 
polynomials and exploiting any extra structure the quandle may have 
(symplectic vector space, $R$-module, etc.); 
generalizing these ideas to the rack case will be the subject of future
papers. 

Replacing the arcs in the combinatorial motivation for the rack 
axioms with semiarcs yields \textit{biracks}, also known 
as invertible switches or \textit{Yang-Baxter sets} (see \cite{FRS,FJK}). 
The birack analogues of the simple and polynomial rack counting invariants 
will be the subject of another future paper.

\medskip

\texttt{Python} code for computing rack counting invariants, reduced rack 
$2$-cocycles with $\mathbb{Z}_n$ coefficients, and rack cocycle invariants 
is available for download at
\texttt{www.esotericka.org}.

\end{document}